\documentclass[10pt]{article}%
\usepackage{amsmath}
\usepackage{amsfonts}
\usepackage{mathrsfs}
\usepackage{amssymb}
\usepackage[linkcolor=black,anchorcolor=black,citecolor=black]{hyperref}
\usepackage{graphicx}
\usepackage[body={15.5cm,21cm}, top=3cm]{geometry}%
\providecommand{\U}[1]{\protect\rule{.1in}{.1in}}
\providecommand{\U}[1]{\protect \rule{.1in}{.1in}}
\newtheorem{theorem}{Theorem}[section]

\newtheorem{definition}[theorem]{Definition}
\newtheorem{example}[theorem]{Example}

\newtheorem{lemma}[theorem]{Lemma}

\newtheorem{remark}[theorem]{Remark}

\newenvironment{proof}[1][Proof]{\noindent \textbf{#1.} }{\  \rule{0.5em}{0.5em}}
\begin{document}

\title{Stochastic differential equations driven by $G$-Brownian motion and ordinary differential equations}
\author{Peng Luo\thanks{School of Mathematics and Qilu Securities Institute for Financial Studies, Shandong University and Department of Mathematics and Statistics, University of Konstanz; pengluo1989@gmail.com}
\and Falei Wang\thanks{School of Mathematics, Shandong University;  flwang2011@gmail.com. Partially supported by Graduate Independent Innovation Foundation
of Shandong University (No. YZC12062). Luo and Wang's research was
partially supported by NSF (No. 10921101) and by the 111
Project (No. B12023)}}
\date{}
\maketitle


\begin{abstract}
In this paper, we show that the integration of a stochastic differential equation driven by  $G$-Brownian motion ($G$-SDE for short) in $\mathbb{R}$ can be reduced to  the integration of an ordinary differential equation (ODE for short) parametrized by a variable in $(\Omega,\mathcal{F})$.  By this result, we obtain a comparison theorem for $G$-SDEs and its applications.
\end{abstract}

\section{Introduction}
Motivated by uncertainty problems, risk measures and the superhedging in finance, Peng systemically established a time-consistent fully nonlinear
expectation theory (see \cite{Peng2004,Peng2005,P09}). As a typical and important case, Peng  introduced the
$G$-expectation theory (see \cite{Peng4, P10}  and the references therein) in 2006.
In the  $G$-expectation framework ($G$-framework for short), the
notion of  $G$-Brownian motion and the corresponding stochastic
calculus of It\^{o}'s type were established. On that basis, Gao \cite{G} and Peng \cite {Peng4}  studied the existence and uniqueness of the solution of $G$-SDE under a standard Lipschitz condition. Moreover,  Lin \cite{L} obtained the existence and uniqueness of the solution of $G$-SDE with reflecting boundary.
For a recent account and development of this
theory we refer the reader to \cite{BL, L1, LQ1, LQ, LP, LY, XZ}.

Under the classical framework, Doss \cite{D} and Huang, Xu and Hu \cite{Huang2} studied the sample solutions of stochastic differential equations, which enables  us to transfer a stochastic differential equation into a set of ordinary differential equations for each sample path. Using the method of sample solutions to SDEs, Huang \cite{Huang1} established a comparison theorem of SDEs.

The aim of this paper is  to study the sample solutions of $G$-SDEs by ODEs parameterized by a variable in basis probability space.
Since $G$-SDE admits a unique solution in the space $M^2_G(0,T)$, the main difficulty is how to prove that the sample solution belongs to this space.
We overcome this problem  through some $G$-stochastic calculus techniques. Then we show that the solution of  $G$-SDE can be represented as a function of both $G$-Brownian motion and a finite variation process. Since we can  use the existing results in the theory of ordinary differential equations directly, this approach provides a powerful tool both in the theoretical analysis and in the practical computation of $G$-SDEs. In particular, we get a new kind of comparison theorem for $G$-SDEs. Moreover, a necessary and sufficient condition for comparison theorem of  $G$-SDEs is also obtained.

This paper is organized as follows: In the next section, we recall some notations and results that we will use in this paper. In section 3, we study the sample solution of $G$-SDE under some strong conditions, then, in section 4, we extend this result to a more general case. Finally in section 5, we establish a new kind of comparison theorem and give its applications.
\section{Preliminaries}
The main purpose of this section is to recall some preliminary
results in $G$-framework which are needed in the sequel.  More
details can be found in  Denis et al \cite{DHP}, Li and Peng \cite{LP}, Lin \cite{L,LY} and Peng \cite{Peng4}.
\subsection{Sublinear expectation}
\begin{definition}
Given a set $\Omega$ and a linear space $\mathcal{H}$
of real valued functions defined on $\Omega$. Moreover, if $X_i \in\mathcal{H}, i = 1, \ldots, d,$
then $\varphi(X_1, \cdots,X_d) \in\mathcal{H}$ for all $ \varphi\in
{C}_{b,Lip}(\mathbb{R}^d)$, where ${C}_{b,Lip}(\mathbb{R}^{d})$ is
the space of all bounded real-valued Lipschitz continuous functions.
 A sublinear expectation $\hat{\mathbb{E}}$ on $\mathcal{H}$ is a
functional $\hat{\mathbb{E}}: \mathcal{H}\rightarrow \mathbb{R}$
satisfying the following
properties: for all $X, Y\in\mathcal{H}$,\\
(a) Monotonicity: if $X\geq Y$, then $\hat{\mathbb{E}}[X]\geq\hat{\mathbb{E}}[Y];$\\
(b) Constant preserving: $\hat{\mathbb{E}}[c]=c$, $\forall \ c\in
\mathbb{R};$\\
(c) Sub-additivity: $\hat{\mathbb{E}}[X+Y]\leq \hat{\mathbb{E}}[X]+\hat{\mathbb{E}}[Y];$\\
(d) Positive homogeneity: $\hat{\mathbb{E}}[\lambda
X]=\lambda\hat{\mathbb{E}}[X]$, $\forall\ \lambda\geq0.$

The triple $(\Omega,\mathcal{H},\hat{\mathbb{E}})$ is called a
sublinear expectation space.  $X \in\mathcal{ H}$ is called a random
variable in $(\Omega,\mathcal{H},\hat{\mathbb{E}})$. We often call
$Y = (Y_1, \ldots, Y_d), Y_i \in\mathcal{ H}$ a $d$-dimensional
random vector in $(\Omega,\mathcal{H},\hat{\mathbb{E}})$.
\end{definition}

\begin{definition}
In a sublinear expectation space
$(\Omega,\mathcal{H},\hat{\mathbb{E}})$, a $n$-dimensional random vector $Y=(Y_1,\ldots,Y_n)$ is said to be independent from  an $m$-dimensional random vector
$X=(X_1,\ldots,X_m)$ under $\hat{\mathbb{E}}$ if for any test
function $\varphi\in{C}_{b,Lip}(\mathbb{R}^{m+n})$
$$\hat{\mathbb{E}}[\varphi(X,Y)]=\hat{\mathbb{E}}[\hat{\mathbb{E}}[\varphi(x,Y)]_{x=X}].$$
\end{definition}
\begin{definition}
Let $X_{1}$ and $X_{2}$ be two $n$-dimensional random vectors defined on sublinear
expectation spaces
$(\Omega_{1},\mathcal{H}_{1},\hat{\mathbb{E}}_{1})$ and
$(\Omega_{2},\mathcal{H}_{2},\hat{\mathbb{E}}_{2})$, respectively.
They are called identically distributed, denoted by $X_{1}\overset{d}{=} X_{2}$,
if
$$\hat{\mathbb{E}}_{1}[\varphi(X_{1})]=\hat{\mathbb{E}}_{2}[\varphi(X_{2})], \ \forall \ \varphi\in {C}_{b,Lip}(\mathbb{R}^n).$$
 $\bar{X}$ is said to be an independent copy of $X$ if $\bar{X}\overset{d}{=}X$ and
$\bar{X}$ is independent from $X$.
\end{definition}
\begin{definition}[$G$-normal distribution]
A random
variable $X$ on a sublinear expectation space
$(\Omega,\mathcal{H},\hat{\mathbb{E}})$ is called (centralized)
$G$-normal distributed if for any $a, b\geq 0$
$$aX+b\bar{X}\overset{d}{=}\sqrt{a^{2}+b^{2}}X,$$
where $\bar{X}$ is an independent copy of $X$. The letter $G$
denotes the function
\[G(a) =\frac{1}{2}(\overline{\sigma}^2 a^+-\underline{\sigma}^2a^-)\]
with $\underline{\sigma}^2 := -\hat{\mathbb{E}}[-X^2]\leq
\hat{\mathbb{E}}[X^2] =: \overline{\sigma}^2.$

\end{definition}

\subsection{$G$-Brownian motion}

\begin{definition}[$G$-Brownian motion] A process $(B_{t}\in\mathcal{H})_{t\geq
0}$ on a sublinear expectation space
$(\Omega,\mathcal{H},\hat{\mathbb{E}})$ is called a $G$-Brownian
motion if the following properties are
satisfied: \\
(a) $B_0=0.$\\
(b) For each $t, s\geq 0$ the increment $B_{t+s}-B_{t}\overset{d}{=} \sqrt{s}X$  and independent from
$(B_{t_{1}},B_{t_{2}},...,B_{t_{n}})$ for each $n\in \mathbb{N}$,
$0\leq t_{1}\leq t_{2}\leq ...\leq t_{n}\leq t$, where $X$ is $G$-normal distributed.
\end{definition}

Denote by
$\Omega=C_{0}(\mathbb{R}^{+})$ the space of all
$\mathbb{R}$-valued continuous paths
$(\omega_{t})_{t\in\mathbb{R}^{+}}$, with $\omega_{0}=0$, equipped
with the distance
$$\rho(\omega^{1},\omega^{2}):=\sum_{i=1}^{\infty}2^{-i}[\max\limits_{t\in[0,i]}|\omega^{1}_t-\omega^{2}_t|\wedge 1].$$
$\mathcal{B}({\Omega})$ is the Borel $\sigma$-algebra of $\Omega$.

For each $t\in[0,\infty)$, we introduce the following spaces.
\\
$\bullet$ $\Omega_{t}:=\{\omega({\cdot\wedge t}):\omega\in\Omega\}$, $\mathcal{F}_{t}:=\mathcal {B}(\Omega_{t})$,
\\
$\bullet$ $L^{0}{(\Omega)}:$ the space of all $\mathcal
{B}(\Omega)$-measurable real functions,
\\
$\bullet$ $L^{0}{(\Omega_{t})}:$ the space of all $\mathcal{F}_{t}$-measurable real functions,
\\
$\bullet$ $B_{b}{(\Omega)}:$ all bounded elements in $L^{0}{(\Omega)}$, $B_{b}{(\Omega_{t})}:=B_{b}{(\Omega)}\cap L^{0}{(\Omega_{t})}$,
\\
$\bullet$
$C_{b}{(\Omega)}:$ all  continuous elements in $B_{b}{(\Omega)}$, $C_{b}{(\Omega_{t})}:=C_{b}{(\Omega)}\cap L^{0}{(\Omega_{t})}.$

In Peng \cite{Peng4}, a $G$-Brownian motion is constructed on a sublinear expectation
space $(\Omega,\mathbb{L}_{G}^1,\hat{\mathbb{E}},(\hat{\mathbb{E}}_t)_{t\geq 0})$, where   $\mathbb{L}_{G}^p(\Omega)$
 is a Banach space
under the natural norm $\|X\|_p=\hat{\mathbb{E}}[|X|^p]^{1/p}$.   In this space the corresponding
canonical process $B_t(\omega) = \omega_t$
 is a $G$-Brownian motion. Denote by $L^p_b(\Omega)$ the completion of $B_b(\Omega)$.
Denis et al.\cite{DHP}  proved that
$L^{p}_b{(\Omega)}\supset\mathbb{L}_{G}^{p}(\Omega)\supset
C_{b}{(\Omega)}$, and there exists a weakly compact family $\mathcal
{P}$ of probability measures defined on $(\Omega, \mathcal
{B}(\Omega))$ such that
$$\hat{\mathbb{E}}[X]=\sup\limits_{P\in\mathcal{P}}E_{P}[X], \ \ X\in\mathbb{L}_{G}^{1}{(\Omega)}.$$

\begin{remark}{\upshape Denis et al. \cite{DHP} gave a concrete set $\mathcal{P}_M$ that represents $\hat{\mathbb{E}}$.
Consider a 1-dimensional Brownian motion $ {B_t}$ on $(\Omega,\mathcal{F},P)$, then
\[
\mathcal{P}_M := \{P_{\theta} : P_{\theta}= P\circ X^{-1},\ X_t = \int^t_0 \theta_sdB_s,\  \theta\in L^2_{\mathcal{F}}([0, T ]; [\underline{\sigma}^2, \overline{\sigma}^2])\}\]
is a set that represents $\hat{\mathbb{E}}$, where $L^2_{\mathcal{F}}([0, T ]; [\underline{\sigma}^2, \overline{\sigma}^2])$ is the collection of all $\mathcal{F}$-adapted
measurable processes with $\underline{\sigma}^2 \leq |\theta(s)|^2 \leq\overline{\sigma}^2$.
}
\end{remark}

Now we introduce the natural choquet capacity
$$c(A):=\sup\limits_{P\in\mathcal{P}}P(A), \ \ A\in\mathcal
{B}(\Omega).$$
\begin{definition}A set $A\subset\mathcal{B}(\Omega)$ is polar if $c(A)=0$.  A
property holds $``quasi$-$surely''$ (q.s.) if it holds outside a
polar set.
\end{definition}
\begin{definition}A real function $X$ on $\Omega$ is said to be quasi-continuous if for each $\varepsilon>0$,
there exists an open set $O$ with $c(O)<\varepsilon$ such that
$X|_{O^{c}}$ is continuous.
\end{definition}

\begin{definition}  We say that  $X:\Omega\mapsto\mathbb{R}$ has a quasi-continuous
version if there exists a quasi-continuous function $Y:\Omega\mapsto\mathbb{R}$ such
that $X = Y$, q.s..
\end{definition}

 Then ${L}_{b}^{p}(\Omega)$ and $\mathbb{L}_{G}^{p}(\Omega)$ can be characterized as
follows:
$${L}_{b}^{p}(\Omega)=\{X\in {L}^{0}(\Omega)|\lim\limits_{N\rightarrow\infty}\mathbb{\hat{E}}[|X|^pI_{|X|\geq N}]=0\}$$
and
$$\mathbb{L}_{G}^{p}(\Omega)=\{X\in {L}^{p}_b(\Omega)|\ \  X\ \text {has a quasi-continuous version}\}.$$
\subsection{$G$-stochastic calculus}

Peng \cite {Peng4} also introduced the related stochastic
calculus of It\^{o}'s type with respect to $G$-Brownian motion (see Li and Peng \cite{LP}, Lin \cite{L} for more general and systematic research).

Let $T\in\mathbb{R}^{+}$ be fixed.
\begin{definition} For each  $p\geq 1$,
consider the following simple type of processes:
\begin{align*}
M_{G}^{0,p}(0,T)=&\{\eta:=\eta_t(\omega)=\sum_{j=0}^{N-1}\xi_{j}(\omega)I_{[t_{j},t_{j+1})}(t)\\
 &\forall\ N>0,\ 0=t_{0}<...<t_{N}=T,\  \xi_{j}\in \mathbb{L}_{G}^p(\Omega_{t_{j}}),\
 j=0,1,2,...,N-1\}.
\end{align*}
Denote by $M_{G}^{p}(0,T)$ the completion of
$M_{G}^{0,p}(0,T)$ under the norm
$$||\eta||_{M^{p}_{G}(0,T)}=|\int_{0}^{T}\hat{\mathbb{E}}[|\eta(t)|^{p}]dt|^{1/p}.$$
\end{definition}

\begin{definition}
For each $\eta\in M_{G}^{0,2}(0,T)$ with the form
$$\eta_{t}(\omega)=\sum_{k=0}^{N-1}\xi_{k}(\omega)I_{[t_{k},t_{k+1})}(t),$$
define
$$I(\eta)=\int_{0}^{T}\eta_sdB_s:=\sum_{k=0}^{N-1}\xi_{k}(B_{t_{k+1}^{N}}-B_{t_{k}^{N}}).$$
The mapping $I: M_{G}^{0,2}(0,T)\mapsto
\mathbb{L}^{2}_{G}(\Omega_{T})$ can be continuously extended to $I:
M_{G}^{2}(0,T)\mapsto \mathbb{L}^{2}_{G}(\Omega_{T})$. For each
$\eta\in M_{G}^{2}(0,T)$, the stochastic integral is defined by
$$I(\eta):=\int_{0}^{T}\eta_sdB_s,\ \ \eta\in M_{G}^{2}(0,T).$$
\end{definition}

Unlike the classical theory, the quadratic variation process of
$G$-Brownian motion $B$ is not always a deterministic process and it
can be formulated in $\mathbb{L}^2_G(\Omega_{t})$ by \[\langle
B\rangle_t : =\lim\limits_{N\rightarrow\infty}\sum\limits_{i=0}^{N-1}(B_{t_{i+1}^N}-B_{t^N_i})^2= B^2_t-2 \int^t_0 B_sdB_s,\]
where $t_i^N=\frac{iT}{N}$ for each integer $N\geq 1$.

\begin{definition}
Define a mapping $M_{G}^{0,1}(0,T)\mapsto\mathbb
{L}^{1}_G(\Omega_{T})$:
$$Q(\eta)=\int_{0}^{T}\eta_sd\langle B\rangle_s:=\sum_{k=0}^{N-1}\xi_{k}[\langle B\rangle_{t_{k+1}^{N}}-\langle B\rangle_{t_{k}^{N}}].$$
Then $Q$ can be uniquely extended to $M_{G}^{1}(0,T)\mapsto \mathbb{L}^{1}_{G}(\Omega_{T})$. We
also denote this mapping by
$$Q(\eta):=\int_{0}^{T}\eta_sd\langle B\rangle_s,\ \ \eta\in M_{G}^{1}(0,T).$$
\end{definition}

 In view of the dual formulation of
$G$-expectation as well as the properties of the quadratic variation
process $\langle B\rangle$ in $G$-framework, Gao \cite{G} obtained the following BDG type
inequalities. \begin{lemma}  For each $p \geq 1$ and $\eta \in
M^p_G(0, T )$,   \[\mathbb{\hat{E}}[ \sup\limits_ {0 \leq t\leq
T} |\int^t_0\eta_sd\langle B\rangle_s|^p]\leq
\bar{\sigma}^{2p}T^{p-1}\int^T_0\mathbb{\hat{E}} [|\eta_s|^p]ds.\]\end{lemma}
\begin{lemma}  Let $p \geq 2$ and $\eta \in M^p_G(0, T )$.
Then there exists some constant $C_p$ depending only on $p$ and $T$ such
that
\[\mathbb{\hat{E}}[ \sup\limits_ {0\leq t\leq T} |\int^t_0\eta_sd
B_s|^p]\leq C_p \mathbb{\hat{E}}[|\int^T_0
|\eta_s|^2ds|^{\frac{p}{2}}].\]\end{lemma}
\section{$G$-Stochastic differential equation }
Let us first recall some notations,\\
 $\bullet$ $C^{n}(\mathbb{R}^d)$: the space of all functions of class $C^{n}$ from $\mathbb{R}^d$ into $\mathbb{R}$,\\
 $\bullet$ $C^{n}_{b,lip}(\mathbb{R}^d)$: the space of all bounded functions of class $C^{n}(\mathbb{R}^d)$ whose partial derivatives of order less than or equal to $n$ are  bounded Lipschtiz continuous functions,\\
 $\bullet$ $C^{n}([0,T]\times\mathbb{R}^d)$: the space of all functions of class $C^{n}$ from $[0,T]\times\mathbb{R}^d$ into $\mathbb{R}$,\\
 $\bullet$ $C^{n}_{b,lip}([0,T]\times\mathbb{R}^d)$: the space of all bounded functions of class $C^{n}([0,T]\times\mathbb{R}^d)$ whose partial derivatives of order less than or equal to $n$ are  bounded Lipschtiz continuous functions.

Consider the following SDE driven by a $1$-dimensional $G$-Brownian motion:
\begin{align}\label{GSDE1}
X_t=X_0+\int^t_0b(s,X_s)ds+\int^t_0h(s,X_s)d\langle B\rangle_s+\int^t_0\sigma(s,X_s)dB_s,\ t\in[0,T],
\end{align}
where the initial condition $X_0\in\mathbb{R}$  is a given constant.

We recall the following assumption.
\begin{description}
\item[(H)] $b, h , \sigma:\Omega\times[0,T]\times\mathbb{R}\rightarrow\mathbb{R}$ are given functions satisfying $ b(\cdot, x), h(\cdot, x), \sigma(\cdot, x)\in M^
2_G(0, T )$ for each $x \in\mathbb{R}$.
Moreover, there exists some constant $K$ such that $|\varphi(t, x) - \varphi(t, y)| \leq K|x -y|$ for each $t \in [0, T ]$,
$x, y\in\mathbb{ R}$, $\varphi = b, h$ and $\sigma$, respectively.
\end{description}
From Peng \cite{Peng4},
\begin{theorem}\label{GSDE2}
Under the assumption \emph{(H)}, there exists a unique solution $X\in M^
2_G(0, T )$ to the stochastic differential equation \eqref{GSDE1}.
\end{theorem}

\begin{remark} {\upshape We remark that there is a potential to extend our results to a much more general setting.
However, in order to focus on the main ideas, in this paper we content ourselves with the case that the coefficients are $1$-dimensional satisfying bounded condition. In particular, by slightly more involved estimates, we can extend our results to the multi-dimensional case without bounded condition.}
\end{remark}

\subsection{A simple case}
In order to explain the main ideas, we  first consider a simple $G$-SDE,
\begin{align}\label{GSDE3}
X_t=X_0+\int^t_0b(X_s)ds+\frac{1}{2}\int^t_0\sigma(X_s)\partial_x\sigma(X_s)d\langle B\rangle_s+\int^t_0\sigma(X_s)dB_s,\ t\in[0,T],
\end{align}
where $\sigma(x)\in C_{b,lip}^{1}(\mathbb{R})$ and $b(x)\in C_{b,lip}(\mathbb{R})$. By Theorem \ref{GSDE2}, $G$-SDE \eqref{GSDE3} admits a unique
solution $X\in M^2_G(0, T )$.

Now consider the following ODE
\begin{equation}\label{ODE1}
\frac{dy}{dx}=\sigma(y),\ \ y(0)=v\in\mathbb{R}.
\end{equation}
The above ODE has a unique solution $y=\varphi(x,v)\in C(\mathbb{R}^2)$. Then,
\begin{equation*}
\partial_x\varphi=\sigma(\varphi),\ \ \ \varphi(0,v)=v.
\end{equation*}
Consequently,
\begin{equation*}
\partial_v\varphi(x,v)=\exp\{\int_{0}^{x}\partial_x\sigma(\varphi(y,v))dy\},\ \ \partial_{xx}^2\varphi(x,v)=(\partial_x\sigma\sigma)(\varphi(x,v)).
\end{equation*}
Next we introduce the following ODE with parameter $\omega$:
\begin{equation}\label{ODE3}
\left\{
\begin{aligned}
dV_t&=\exp\{-\int_{0}^{B_t(\omega)}\partial_x\sigma(\varphi(y,V_t))dy\}b(\varphi(B_t(\omega),V_t))dt,\\
V_0&=X_0.
\end{aligned}
\right.
\end{equation}
For every fixed $\omega$, recalling Cauchy--Lipschitz theorem, the equation \eqref{ODE3} has a unique solution $V_t = V_t(\omega)$ and $ V_t$ is a continuous finite variation process.
Moreover, $V_t(\omega)$ is a continuous function on $(\Omega,\rho)$.

The following result is important
in our future discussion.
\begin{lemma}\label{lem5}
For any $p\geq 0$, there exists a constant $C_p$ depending only on $p$ such that,
\[
\hat{\mathbb{E}}[\sup\limits_{0\leq t\leq T}e^{p|B_t|}]\leq C_p.
\]
\end{lemma}
\begin{proof}
For any $p\geq 0$, we have
\[\hat{\mathbb{E}}[\sup\limits_{0\leq t\leq T}e^{p|B_t|}]\leq \sum\limits_{n=0}\hat{\mathbb{E}}[\sup\limits_{0\leq t\leq T}\frac{|pB_t|^n}{n!}].\]
Applying Doob's maximal inequality yields that
\[
\hat{\mathbb{E}}[\sup\limits_{0\leq t\leq T}|pB_t|^n]\leq (1+\frac{1}{n-1})^n \hat{\mathbb{E}}[|pB_T|^n].
\]
By Exercise 1.7 in Chapter 3 of Peng \cite{Peng4}, one can show that for some constant $C'_p$ depending only on $p$,
\[\sum\limits_{n=0}\hat{\mathbb{E}}[\frac{|pB_T|^n}{n!}]\leq C'_p.
\]
Since $\lim\limits_{n}(1+\frac{1}{n-1})^n=e$,  we can find some constant $C_p$ depending only on $p$ such that,
\[
\hat{\mathbb{E}}[\sup\limits_{0\leq t\leq T}e^{p|B_t|}]\leq C_p,
\]
which is the desired result.
\end{proof}
\begin{lemma} \label{lw4} For each $p\geq 1$, $V_t\in\mathbb{L}_G^p(\Omega_t)$.
Moreover,  there exists some constant $C_p$ depending only on $p$ such that for each $s\leq t\in[0,T]$,
\[
\hat{\mathbb{E}}[|T^{V}_t-T^V_s|^p]\leq C_p|t-s|^p,
\]
where $T^V$ is the total variation process of $V$.
\end{lemma}
\begin{proof}
By equation \eqref{ODE3},
\[
V_t=V_0+\int^t_0\exp\{-\int_{0}^{B_u(\omega)}\partial_x\sigma(\varphi(y,V_u))dy\}b(\varphi(B_u(\omega),V_u))du.
\]
Denote by  $C_p$  a constant depending only on $p$, which is
allowed to change from line to line. Then  applying Lemma \ref{lem5}, we conclude
\begin{align*}
\hat{\mathbb{E}}[\sup\limits_{0\leq t\leq T}|V_t|^p]&\leq C_p\hat{\mathbb{E}}[|V_0|^p+\int^T_0\exp\{-p\int_{0}^{B_u(\omega)}\partial_x\sigma(\varphi(y,V_u))dy\}du])
\\
& \leq C_p(|V_0|^p+C\hat{\mathbb{E}}[\sup\limits_{0\leq t\leq T}e^{Cp|B_t|}]) \leq C_p.
\end{align*}
Since $V_t(\omega)$ is a continuous function on $(\Omega,\rho)$, recalling the pathwise description of $\mathbb{L}_G^p(\Omega_t)$,
$V_t\in\mathbb{L}_G^p(\Omega_t)$ for each $p\geq 1$.

Note that
\[
T^V_t=\int^t_0\exp\{-\int_{0}^{B_u}\partial_x\sigma(\varphi(y,V_u))dy\}|b(\varphi(B_u,V_u))|du,
\]
applying Lemma \ref{lem5} again, we obtain
for each $s\leq t\in[0,T]$,
\[
\hat{\mathbb{E}}[|T^{V}_t-T^V_s|^p]\leq C_p|t-s|^p,
\]
which  completes the proof.
\end{proof}

By Lemma \ref{lw4}, we deduce that $\varphi(B_{t},V_t)\in M^2_G(0,T)$.
Since $\varphi$ satisfies the conditions of Theorem \ref{lw5}, applying $G$-It\^{o} formula, we get
\begin{flalign*}
   \begin{split}
    d\varphi(B_{t},V_t)&=\partial_x
\varphi(B_{t},V_t) dB_t +\partial_v\varphi(B_{t},V_t)
dV_t+\frac{1}{2}\partial^2_{xx}\varphi(B_{t},V_t)d\langle
B\rangle_t\\
    &=b(\varphi(B_{t},V_t))dt+\frac{1}{2}\partial_x\sigma(\varphi(B_{t},V_t))\sigma(\varphi(B_{t},V_t))d\langle B\rangle_t+\sigma(\varphi(B_{t},V_t))dB_t.
    \end{split}
\end{flalign*}
Consequently, $X_t=\varphi(B_{t},V_t)$ is the unique $M^2_G(0,T)$-solution
of $G$-SDE (\ref{GSDE3}).

\subsection{The general case}
In this section,  we will extend the above result to a more general case, where all the
coefficients are functions in $t, B_t$ and $x$.  Assume $b(t,x,y),h(t,x,y)\in C_{b,lip}([0,T]\times\mathbb{R}^2)$ and $\sigma(t,x,y)\in C_{b,lip}^{1}([0,T]\times\mathbb{R}^2)$.
It is obvious $G$-SDE
\begin{align}\label{GSDE4}
X_t=X_0+\int^t_0b(s,B_s,X_s)ds+\int^t_0h(s,B_s,X_s)d\langle B\rangle_s+\int^t_0\sigma(s,B_s,X_s)dB_s,\ t\in[0,T]
\end{align}
has a unique solution $X\in M^2_G(0,T)$.

Then the following ODE
\begin{equation}\label{zz1}
\frac{dy}{dx}=\sigma(t,x,y),\ \ y(t,0)=v\in\mathbb{R}
\end{equation}
admits a unique solution $y=\varphi(t,x,v)\in C([0,T]\times\mathbb{R}^2)$. Moreover, we can get
\[
\partial_v\varphi(t,x,v)=\exp\{\int_0^{x}\partial_y\sigma(t,u,\varphi(t,u,v))du\}
\]
and
\[
\partial_t\varphi(t,x,v)=\exp\{\int_0^{x}\partial_y\sigma(t,z,\varphi(t,z,v))dz\}(\int_0^{x}\partial_t\sigma(t,u,\varphi(t,u,v))e^{-\int_0^{u}\partial_y\sigma(t,z,\varphi(t,z,v))dz}du).
\]

Set
\begin{align*}
g(t,x,v):=\partial_v&\varphi^{-1}(t,x,v)(b(t,x,\varphi(t,x,v))-\partial_t\varphi(t,x,v)),\\
f(t,x,v):=\partial_v\varphi^{-1}(t,x,v)&(h(t,x,\varphi(t,x,v))-\frac{1}{2}(\partial_x\sigma+\partial_y\sigma\sigma)(t,x,\varphi(t,x,v))).
\end{align*}
Then consider the following initial value problem with parameter $\omega$:
\begin{equation}\label{ODE4}
\left\{
\begin{aligned}
dV_t&=g(t,B_t(\omega),V_t)dt+f(t,B_t(\omega),V_t)d\langle B\rangle_t(\omega),\\
V_0&=X_0.
\end{aligned}
\right.
\end{equation}
Note that $\langle B\rangle_t$ is a continuous finite variation process, then the ODE \eqref{ODE4}  has a unique solution $V=V_t(\omega)$ and
 $ V_t$ is a continuous finite variation process. Since $\langle B\rangle_t(\omega)$ is not always a deterministic process,
 in general we can not get $V_t(\omega)$ is a continuous function on $(\Omega,\rho)$ as the above section. However, we also have the following result.

\begin{lemma} \label{lw6} For each $p\geq 1$,
 there exists some constant $C_p$ depending only on $p$ such that, for each $s\leq t\in[0,T]$,
\[
\hat{\mathbb{E}}[|T^{V}_t-T^V_s|^p]\leq C_p|t-s|^p,
\]
where $T^V$ is the total variation process of $V$.
\end{lemma}
\begin{proof}
 The proof is immediate in light of Lemma \ref{lw4}.
\end{proof}

Now we shall give the main result of this section.
\begin{theorem}
 Assume  $b(t,x,y),h(t,x,y)\in C_{b,lip}([0,T]\times\mathbb{R}^2)$ and $\sigma(t,x,y)\in C_{b,lip}^{1}([0,T]\times\mathbb{R}^2)$, then for each $p\geq 1$, $V_t\in\mathbb{L}^p_G(\Omega_t)$ and $\varphi(t,B_t,V_t)$ is the unique $M_G^2(0,T)$-solution of $G$-SDE \eqref{GSDE4}.
\end{theorem}
\begin{proof}
It is obvious $V_t\in L^p_b(\Omega_t)$.
Then  applying Theorem \ref{lw5}, we obtain q.s.
\begin{flalign*}
   \begin{split}
    d\varphi(t,B_t,V_t)&=\partial_t\varphi(t,B_t,V_t)dt+\partial_x
\varphi(t,B_{t},V_t) dB_t +\partial_v\varphi(t,B_{t},V_t)
dV_t+\frac{1}{2}\partial^2_{xx}\varphi(t,B_{t},V_t)d\langle
B\rangle_t\\&=b(t,B_t,\varphi(t,B_t,V_t))dt+h(t,B_t,\varphi(t,B_t,V_t))d\langle
B\rangle_t+\sigma(t,B_t,\varphi(t,B_t,V_t))dB_t.
    \end{split}
\end{flalign*}
By a standard argument, there exists some constant $C$ such that,
\[
\hat{\mathbb{E}}[|\varphi(t,B_t,V_t)-X_t|^2]\leq C\int^t_0\hat{\mathbb{E}}[|\varphi(s,B_s,V_s)-X_s|^2]ds.
\]
Applying  Gronwall's lemma, we obtain $\varphi(t,B_t,V_t)=X_t $, q.s..

By the uniqueness of solution of ODE (\ref{zz1}),
\begin{equation*}
    v=\varphi(t,-x,\varphi(t,x,v)),
\end{equation*}
thus, $V_t=\varphi(t,-B_t,X_t)$ q.s.. In particular, $V_t$ has a quasi-continuous version and $V_t\in\mathbb{L}^p_G(\Omega_t)$. The proof is completed.
\end{proof}
\section{$G$-diffusion process}
The objective of this section is to  remove the condition  that $\sigma$ is continuously differentiable and to obtain a more general result on this topic. By an approximation approach, we can also represent the solution of $G$-SDE as a function of $B_t$ and a continuous finite variation process $V_t$ as the above section.

\begin{theorem}\label{lw3}
If $b,\sigma,h\in C_{b,lip}(\mathbb{R})$, then there exists a unique continuous finite variation process $V_t\in \mathbb{L}_G^p(\Omega_t)$ for each $p\geq 1$ such that,
\[
X_t=\varphi(B_t,V_t),
\]
where $\varphi$ is the solution of the ODE:
\begin{equation*}
\partial_x\varphi(x,v)=\sigma(\varphi(x,v)),\ \ \ \varphi(0,v)=v.
\end{equation*}
Moreover if $\sigma\in C^1_{b,lip}(\mathbb{R})$, then for q.s. $\omega$, $V_t(\omega)$ is the solution of the following ODE:
\begin{equation}\label{ODE5}
\left\{
\begin{aligned}
dV_t&=\exp\{-\int_{0}^{B_t}\partial_x\sigma(\varphi(y,V_t))dy\}[b(\varphi(B_t,V_t))dt+(h(\varphi(B_t),V_t))-\frac{1}{2}
\partial_x\sigma\sigma(\varphi(B_t,V_t))d\langle B \rangle_t)],\\
V_0&=X_0.
\end{aligned}
\right.
\end{equation}

\end{theorem}
\begin{proof}If $\sigma\in C^1_{b,lip}(\mathbb{R})$, then the theorem holds true.
If $\sigma\in C_{b,lip}(\mathbb{R})$,  one can define \[\sigma^n(x):=\int_{\mathbb{R}}\sigma(y)\rho_n(y-x)dy=\int_{\mathbb{R}}\sigma(y+x)\rho_n(y)dy,\]
where $\rho_n$ is a nonnegative $C^{\infty}$ function defined on $\{x: |x|\leq \frac{1}{n}\}$ with $\int_{\mathbb{R}}\rho_n(y)dy=1$.
From this definition, we conclude that
\[
|\sigma^{n}(x)-\sigma(x)|\leq\int_{\mathbb{R}}|\sigma(y+x)-\sigma(x)|\rho_{n}(y)dy\leq\int_{\mathbb{R}}K|y|\rho_n(y)dy\leq\frac{K}{n},
\]
where $K$ is the Lipschitz coefficient of $b,h$ and $\sigma$.

For each $n$, it is obvious $\sigma^n \in C^1_{b,lip}(\mathbb{R})$. Thus, $ X^n_t:=\varphi^n(B_t,V^n_t)$ is the solution of $G$-SDE:
\begin{align*}
X^n_t=X_0+\int^t_0b(X^n_s)ds+\int^t_0h(X^n_s)d\langle B\rangle_s+\int^t_0\sigma^n(X^n_s)dB_s, \ t\in[0,T],
\end{align*}
where $\varphi^n$ satisfies
\begin{equation*}
\partial_x\varphi^n(x,v)=\sigma^n(\varphi^n(x,v)),\ \ \ \varphi^n(0,v)=v
\end{equation*}
and \begin{equation*}
\left\{
\begin{aligned}
dV^n_t&=\exp\{-\int_{0}^{B_t}\partial_x\sigma^{n}(\varphi^n(y,V^n_t))dy\}[b(\varphi^n(B_t,V^n_t))dt+(h(\varphi^n(B_t,V^n_t))\\
&\ \ \ \ \ \ \ \ \ \ -\frac{1}{2}
\partial_x\sigma^{n}\sigma^{n}(\varphi^n(B_t,V^n_t))d\langle B \rangle_t)],\\
V^n_0&=X_0.
\end{aligned}
\right.
\end{equation*}

For each $n$, there exists  some constant $C$ depending only on $T$ and  $K$ such that,
\begin{align*}
\hat{\mathbb{E}}[\sup\limits_{t\in[0,T]}|X^n_t-X_t|^2]\leq \frac{C}{n^2}.
\end{align*}
Indeed,  applying BDG inequalities, we obtain for some constant $C$, which is allowed to change from line to
line,
\begin{align*}
\hat{\mathbb{E}}[\sup\limits_{t\in[0,T]}|X^{n}_t-X_t|^2]&\leq \hat{\mathbb{E}}[\sup\limits_{t\in[0,T]}|\int^t_0b(X^n_s)-b(X_s)ds+\int^t_0h(X^n_s)-h(X_s)d\langle B\rangle_s+\int^t_0\sigma^n(X^n_s)-\sigma(X_s)dB_s|^2]\\
&\leq C\hat{\mathbb{E}}[(\int_{0}^T K|X^{n}_t-X_t|ds)^2+(\int_{0}^T K|X^{n}_t-X_t|d\langle B\rangle_s)^2+(\int_{0}^T (K|X^{n}_t-X_t|+\frac{K}{n})dB_s)^2]\\
&\leq C(\frac{1}{n^2}+\int_{0}^{T}\hat{\mathbb{E}}[|X^{n}_t-X_t|^2]dt)\\
&\leq C(\frac{1}{n^2}+\int_{0}^{T}\hat{\mathbb{E}}[\sup\limits_{s\in[0,t]}|X^{n}_s-X_s|^2]dt).
\end{align*}
By Gronwall's lemma, we can get the desired result.
Moreover, choosing a subsequence if necessary, $X^n\rightarrow X$ uniformly in $[0,T]$ q.s..

For each $v_1,v_2,x\in\mathbb{R}$,
\[
|\varphi^n(x,v_1)-\varphi(x,v_2)|\leq |\varphi^n(x,v_1)-\varphi^n(x,v_2)|+|\varphi^n(x,v_2)-\varphi(x,v_2)|.
\]
Applying Taylor formula yields that
\[
|\varphi^n(x,v_1)-\varphi^n(x,v_2)|\leq|\partial_{v}\varphi^{n}(x,v^*)||v_1-v_2|\leq|v_1-v_2|e^{C|x|}.
\]
By the definitions of $\varphi^n$ and $\varphi$,  we obtain
\[
|\varphi^n(x,v_2)-\varphi(x,v_2)|\leq|\int_{0}^{x}\sigma^n(\varphi^n(s,v_2))-\sigma(\varphi(s,v_2))ds|\leq\int_{0}^{|x|}K(|\varphi^n(s,v_2)-\varphi(s,v_2)|+\frac{1}{n})ds.
\]
From Gronwall's lemma, we conclude  for some constant $C$
\[
|\varphi^n(x,v_1)-\varphi(x,v_2)|\leq C(|v_1-v_2|+\frac{|x|}{n})e^{C|x|}.
\]
Define $V_t:=\varphi(-B_t,X_t)$, thus $X_t=\varphi(B_t,V_t)$ and $V_t$ has a quasi-continuous version.
Moreover, \begin{align*}
\lim\limits_{n\rightarrow\infty}\sup\limits_{t\in[0,T]}|V^n_t-V_t|&=\lim\limits_{n\rightarrow\infty}\sup\limits_{t\in[0,T]}|\varphi^n(-B_t,X^n_t)-\varphi(-B_t,X_t)|\\
&\leq C\lim\limits_{n\rightarrow\infty}(\sup\limits_{t\in[0,T]}|X^n_t-X_t|+\sup\limits_{t\in[0,T]}\frac{|B_t|}{n})e^{C\sup\limits_{t\in[0,T]}|B_t|}=0.
\end{align*}
Since for each $n$ and $t,s\in[0,T]$, there exists some constant $C$ such that \[
|V^n_t-V^n_s|\leq C\sup\limits_{t\in[0,T]}e^{C|B_t|}|t-s|.
\]
Thus\[
|V_t-V_s|\leq C\sup\limits_{t\in[0,T]}e^{C|B_t|}|t-s|.
\]

By the pathwise description of $\mathbb{L}_G^p(\Omega_t)$, $V_t\in \mathbb{L}_G^p(\Omega_t)$ for each $p\geq 1$ and the proof is completed.
\end{proof}

In general, we can also get
\begin{theorem}\label{lw8}
If $b,\sigma,h\in C_{b,lip}([0,T]\times\mathbb{R}^2)$, then there exists a unique continuous finite variation process  $V_t\in \mathbb{L}_G^p(\Omega_t)$ for each $p\geq 1$ such that
\[
X_t=\varphi(t,B_t,V_t),
\]
where $\varphi$ is given by equation \eqref{zz1}.
Moreover if $\sigma\in C_{b,lip}^1([0,T]\times\mathbb{R}^2)$, then for q.s. $\omega$, $V_t$ is the solution of ODE \eqref{ODE4}.
\end{theorem}
\section{Comparison Theorem for $G$-SDEs}
In the above sections, we establish the relations between $G$-SDEs and ODEs. From these results, we shall study the comparison theorem for $G$-SDEs.  We refer to
 Lin \cite{LQ} for some sufficient condition under which a comparison theorem  for $G$-SDEs is also  obtained  by virtue of a stochastic calculus approach.

We begin with a lemma, which is from \cite{Huang1}.
\begin{lemma}\label{lw1}
Assume that two functions $f(t,x)$ and $\tilde{f}(t,x)$ are defined on $\mathbb{R}^2$, satisfying the Carath\'{e}odory condintions, that is, they are measurable in $t$, continuous in $x$ and dominated by a locally integrable function $m_t$ in $\mathbb{R}^2$. Let $(t_0,x_0)$, $(t_0,\tilde{x}_0)$ be two points in $\mathbb{R}^2$ such that $x_0\leq\tilde{x}_0$. Moreover, $x_t$ is a  solution to the initial value problem
\begin{align*}
dx_t=f(t,x_t)dt,~~~x_{t_0}=x_0,
\end{align*}
and $\tilde{x}_t$ is the maximal solution to the problem
\begin{align*}
d\tilde{x}_t=\tilde{f}(t,\tilde{x}_t)dt,~~~\tilde{x}_{t_0}=\tilde{x}_0.
\end{align*}
If the inequality
\begin{align*}
(t-t_0)f(t,x)\leq (t-t_0)\tilde{f}(t,x)
\end{align*}
holds a.e. in $\mathbb{R}^2$, then $x(t)\leq\tilde{x}(t)$ for every $t$ in the common interval of existence of the solutions $x_t$ and $\tilde{x}_t$.
\end{lemma}
Then we have the following comparison theorem.
\begin{theorem}\label{Com}
Let $b(t,x,y),h(t,x,y)\in C_{b,lip}([0,T]\times\mathbb{R}^2)$ and $\sigma(t,x,y)\in C_{b,lip}^{1}([0,T]\times\mathbb{R}^2)$ be given.  If there exists three functions $\tilde{\sigma}$, $\tilde{f}$ and $\tilde{g}$ satisfying the Carath\'{e}odory conditions and the inequalities
\begin{align*}
x\sigma(t,x,y)\leq x\tilde{\sigma}(t,x,y),\
2G(f(t,x,y)-\tilde{f}(t,x,y))+ g(t,x,y)-\tilde{g}(t,x,y)\leq 0~~\text{in}\ \ [0,T]\times\mathbb{R}^2.
\end{align*}
Then for the unique solution $X_t$ of SDE \eqref{GSDE4}
\[
X_t\leq\tilde{\varphi}(t,B_t,\tilde{V}_t)
\]
holds for q.s. $\omega$ and every $t$ in the common interval where both sides are defined. Here $\tilde{\varphi}$ and $\tilde{V}$ are the maximal solutions to the problems
\begin{equation*}
\frac{d\tilde{\varphi}}{dx}=\tilde{\sigma}(t,x,\tilde{\varphi}),\ \ \tilde{\varphi}(t,0,v)=v\in\mathbb{R}.
\end{equation*}
and
\begin{equation*}
d\tilde{V}_t=\tilde{g}(t,B_t,\tilde{V}_t)dt+\tilde{f}(t,B_t(\omega),\tilde{V}_t)d\langle B\rangle_t,~~\tilde{V}_0=\tilde{X}_0
\end{equation*}
with $X_0\leq\tilde{X}_0$, respectively.
\end{theorem}
\begin{proof}
According to the Lemma \ref{lw1}, we get
\[
\varphi(t,x,v)\leq\tilde{\varphi}(t,x,\tilde{v})
\]
provided $v\leq\tilde{v}$.
From \cite{DHP} or \cite{snz},  $d\langle B\rangle_t = \hat{\alpha}_t(\omega)dt$, where $\hat{\alpha}$ is well defined
for each $\omega$ and q.s. takes value  in $[\underline{\sigma}^2,\overline{\sigma}^2]$. Since $2G(f(t,x,v)-\tilde{f}(t,x,v))+g(t,x,v)-\tilde{g}(t,x,v)\leq 0$,
we obtain $$(f(t,x,v)-\tilde{f}(t,x,v))\hat{\alpha}_t+g(t,x,v)-\tilde{g}(t,x,v)\leq 0.$$
Then applying  Lemma \ref{lw1} again, we also have $V_t\leq\tilde{V}_t$ in the common interval where both sides are defined, which is the desired result.
\end{proof}

Now we consider some examples of its applications.
\begin{example} \label{lw12} {\upshape
Consider two $G$-SDEs with the same diffusion coefficient $\sigma$:
\begin{equation*}
\left\{\begin{aligned}
dX^{i}_t&=b^{i}(t,B_t,X^{i}_t)dt+h^{i}(t,B_t,X^{i}_t)d\langle B\rangle_t+\sigma(t,B_t,X^{i}_t) dB_t,\\
X^{i}_0&=X^{i}_0\qquad(i=1,2),
\end{aligned}
\right.
\end{equation*}
where $\sigma,~b^1,~b^2,~h^1, h^2$ satisfy the conditions in Theorem \ref{Com} and $X_{0}^1\leq X_{0}^2$, $b^1- b^2+2G(h^1- h^2)\leq 0$. Denote:
\[
g^{i}(t,x,v)=\partial_v\varphi^{-1}(t,x,v)(b^{i}(t,x,\varphi(t,x,v))-\partial_t\varphi(t,x,v)),
\]
\[
f^{i}(t,x,v)=\partial_v\varphi^{-1}(t,x,v)(h^{i}(t,x,\varphi(t,x,v))-\frac{1}{2}(\partial_x\sigma+\partial_y\sigma\sigma)(t,x,\varphi(t,x,v))),\qquad (i=1,2)
\]
One can easily show  that
\[
g^1-g^2+2G(f^1- f^2)\leq 0.
\]
Applying Theorem \ref{Com}, we obtain $X^{1}_t\leq X^{2}_t$  q.s..}
\end{example}
\begin{remark}{\upshape
In Example \ref{lw12}, we can also assume that $\sigma\in C_{b,lip}([0,T]\times\mathbb{R}^2)$. Indeed, applying Theorem \ref{lw3}, there exists a sequence $X^{i,n}\rightarrow X^i$ uniformly in $[0, T]$. Then we conclude $X^1_t\leq X^2_t$ from $X^{1,n}_t\leq X^{2,n}_t$ for each $t\in[0,T]$.}
\end{remark}
In particular, we obtain a necessary and sufficient condition for comparison theorem of $1$-dimensional $G$-SDEs.
\begin{theorem}
Consider two $G$-SDEs:
\begin{equation}
\left\{\begin{aligned}
dX^{i,x_i}_t&=b^{i}(X^{i,x_i}_t)dt+h^{i}(X^{i,x_i}_t)d\langle B\rangle_t+\sigma^i(X^{i,x_i}_t) dB_t,\\
X^{i,x_i}_0&=x_i,
\end{aligned}
\right.
\end{equation}
where $\sigma^i,b^i,h^i\in C_{b,lip}(\mathbb{R})$ and $i\in\{1,2\}$, then for each  $x_1\leq x_2$, $X^{1,x_1}_t\leq X^{2,x_2}_t$ if and only if \[b^1(x)-b^2(x)+2G(h^1(x)-h^2(x))\leq 0, \ \ \ \sigma^1(x)=\sigma^2(x), \ \ \forall x\in\mathbb{R}.\]
 \end{theorem}
 \begin{proof}
We shall only have to prove that from  $X^{1,x}_t\leq X^{2,x}_t$ for each $x\in\mathbb{R}$, we infer
that $$b^1(x)-b^2(x)+2G(h^1(x)-h^2(x))\leq 0, \ \ \sigma^1(x)=\sigma^2(x).$$

By $X^{1,x}_t\leq X^{2,x}_t$, we get
\begin{align*}
&\int^t_0b^{1}(X^{1,x}_s)ds+\int^t_0h^{1}(X^{1,x}_s)d\langle B\rangle_s+\int^t_0\sigma^1(X^{1,x}_s) dB_s\\
&\leq \int^t_0b^{2}(X^{2,x}_s)ds+\int^t_0h^{2}(X^{2,x}_s)d\langle B\rangle_s+\int^t_0\sigma^2(X^{2,x}_s) dB_s.
\end{align*}
Set $\alpha^i_s=b^{i}(X^{i,x}_s)-b^{i}(x)$, $\beta^i_s=h^{i}(X^{i,x}_s)-h^{i}(x)$ and $\gamma^i_s=\sigma^{i}(X^{i,x}_s)-\sigma^{i}(x)$.
From Peng \cite{Peng4}), there exists some constant $C$ such that,
\[
\hat{\mathbb{E}}[\sup\limits_{s\in[0,t]}(|\alpha^i_s|^2+|\beta^i_s|^2+|\gamma^i_s|^2)]\leq Ct.
\]
For each  $t\in[0,T]$, we have
\begin{align}\label{3lw1}
&(b^{1}(x)-b^{2}(x))t+(h^{1}(x)-h^{2}(x))\langle B\rangle_t+(\sigma^{1}(x)-\sigma^{2}(x))B_t\nonumber\\&
\leq \int^t_0(\alpha^2_s-\alpha^1_s)ds+\int^t_0(\beta^2_s-\beta^1_s)d\langle B\rangle_s+\int^t_0(\gamma^2_s-\gamma^1_s) dB_s.
\end{align}

Applying BDG inequalities, we can find some constant  $C$ so that
\[
\hat{\mathbb{E}}[|\int^t_0\gamma^i_s dB_s|^2]\leq C\int^t_0\hat{\mathbb{E}}[|\gamma^i_s|^2]ds\leq Ct^2.
\]
Thus q.s.
\[\lim\limits_{t\downarrow 0}\int^t_0\frac{\gamma^i_s}{\sqrt{t}}dB_s=0.
\]
In a similar way we can also obtain q.s.\[ \lim\limits_{t\downarrow 0}\frac{1}{\sqrt{t}}\int^t_0\alpha^i_sds=0,\ \ \
\lim\limits_{t\downarrow 0}\frac{1}{\sqrt{t}}\int^t_0\beta^i_sd\langle B\rangle_s=0.\]

Recalling that
\[
c(\limsup\limits_{t\downarrow 0}\frac{B_t}{\sqrt{t}}=+\infty)=1, \ \ c(\liminf\limits_{t\downarrow 0}\frac{B_t}{\sqrt{t}}=-\infty)=1,
\]
then there exists a subset $\Omega_0\subset\Omega$ with $c(\Omega_0)=1$, such that for each $\omega\in\Omega_0$, we can find a sequence $(r_n:=r_n(\omega))$ so that $\lim\limits_{r_n\downarrow 0}\frac{B_{r_n}}{\sqrt{r_n}}=+\infty$.
By equation  \eqref{3lw1}, we derive that
\begin{align*}
&(\sigma^{1}(x)-\sigma^{2}(x))\lim\limits_{r_n\downarrow 0}\frac{B_{r_n}}{\sqrt{r_n}}\leq 0,
\end{align*}
Consequently, $\sigma^1(x)\leq \sigma^2(x)$.  Similarly we can prove $\sigma^1(x)\geq \sigma^2(x)$, then,
\[
\sigma^1(x)=\sigma^2(x).
\]
Finally, taking  expectation on both sides of equation \eqref{3lw1} yields
\begin{align*}
b^{1}(x)-b^{2}(x)+2G(h^{1}(x)-h^{2}(x))
\leq \lim\limits_{t\rightarrow 0}\frac{1}{t}\hat{\mathbb{E}}[\int^t_0(\alpha^2_s-\alpha^1_s)ds+\int^t_0(\beta^2_s-\beta^1_s)d\langle B\rangle_s]=0,
\end{align*}
which completes the proof.
 \end{proof}

\begin{remark}{\upshape
Let $ \sigma^1=\sigma^2=b^1=h^2=0,b^2=\overline{\sigma}^2,h^1=1$,  one can show that
$b^1(x)-b^2(x)+2G(h^1(x)-h^2(x))\leq 0$ and $X^{1.x}_t=x+\langle B\rangle_t\leq x+ \overline{\sigma}^2t=X^{2,x}_t$ q.s.. Thus $b^1(x)-b^2(x)+2G(h^1(x)-h^2(x))\leq 0$ does not
imply $b^1(x)\leq b^2(x)$ and $h^1(x)\leq h^2(x)$.}
\end{remark}

\begin{example}{\upshape
Consider two $G$-SDEs with different diffusion coefficients:
\begin{equation*}
\left\{
\begin{aligned}
dX_{t}^i&=\sigma_{i}(X_t^{i})dB_t+\frac{1}{2}\sigma_{i}(X_{t}^{i})\sigma_{i}^{\prime}(X_{t}^i)d\langle B\rangle_t,\\
X_{0}^i&=X_{0}^i \qquad (i=1,2),
\end{aligned}
\right.
\end{equation*}
where $\sigma_i>0$ and $\sigma_i\in C^{1}_{b,lip}(\mathbb{R})$.
Consider the following initial value problems:
\[
\frac{d\varphi_i}{dx}=\sigma_{i}(\varphi_i),\quad \varphi_{i}(0)=v_i \quad (i=1,2).
\]
Clearly, the solutions $\varphi_{i}(x,v_i)$ satisfy the equalities
\[
\int_{v_1}^{\varphi_{1}(x,v_1)}\frac{ds}{\sigma_{1}(s)}=x=\int_{v_2}^{\varphi_{2}(x,v_2)}\frac{ds}{\sigma_{2}(s)}.
\]
Note that $b_i\equiv 0$, we can obtain $V_i\equiv X_{0}^i$ by equation \eqref{ODE3}. Hence, if for every $x\in\mathbb{R}$ the inequality
\[
\int_{X_{0}^1}^{x}\frac{dy}{\sigma_{1}(y)}\geq\int_{X_{0}^2}^{x}\frac{dy}{\sigma_{2}(y)}
\]
holds, then $\varphi_1(x,X_{0}^1)\leq\varphi_2(x,X_{0}^2)$ and therefore q.s.
\[
X_{t}^1=\varphi_1(B_t,X_{0}^1)\leq\varphi_2(B_t,X_{0}^2)=X_{t}^2.
\]}
\end{example}
\begin{example} \label{lw2} {\upshape
Consider the following  $G$-SDE:
\begin{equation*}
\left\{\begin{aligned}
dX_t&=b(X_t)dt+h(X_t)d\langle B\rangle_t+\sigma(X_t) dB_t,\\
X_0&=X_0,
\end{aligned}
\right.
\end{equation*}
where $b, h\in C_{b,lip}(\mathbb{R})$ and $\sigma\in C^1_{b,lip}(\mathbb{R})$.
Then we have for some constant $C$
\[
|\sigma(x)|\leq C, \ \, |g(x,v)|\leq Ce^{C|x|},\ \, |f(x,v)|\leq Ce^{C|x|}.
\]
Let $\tilde{\sigma}(x)=Csgn(x)$, $\tilde{g}(x,v)=Ce^{C|x|}$ and $\tilde{f}(x,v)=Ce^{C|x|}$, combining these three inequalities and using Theorem \ref{Com}, we obtain an asymptotic estimation for the paths of $G$-It\^{o} diffusion process  $X_t$, for q.s. $\omega$,
\[
X_t\leq C|B_t|+C\int^t_0e^{C|B_s|}ds+X_0.
\]
A symmetric argument shows that, for q.s. $\omega$,
\[
X_t\geq -C|B_t|-C\int^t_0e^{C|B_s|}ds+X_0.
\]}
\end{example}
\section{Appendix}
$G$-It\^{o} formula for a $G$-It\^{o} process was obtained by Peng \cite {Peng4}  and improved by Gao \cite{G}, Zhang et al \cite{XZ} in $\mathbb{L}_G^p(\Omega)$. Li and Peng \cite{LP}  significantly improved the previous ones for a general $C^{1,2}$-function in a lager space ${L}^p_b(\Omega)$ instead of  $\mathbb{L}_G^p(\Omega)$ (see also  Lin \cite{L}, \cite{LY}). For reader's convenience, we give the following $G$-It\^{o} formula. Indeed, it can be viewed as a special case of Theorem 2.33 of Lin \cite{LY}.

For each $0\leq t\leq T$, consider a $G$-It\^{o} process:
\[
X_t=X_0+\int_{0}^{t}f_{u}du+\int_{0}^{t}h_{u}d\langle B\rangle _u+\int_{0}^{t}g_{u}dB_u.
\]
\begin{theorem}\label{lw5}
Suppose $\varphi\in C([0,T]\times\mathbb{R}^2)$ satisfies for each $t_1,t_2\in[0,T]$, $x_1,x_2,v_1,v_2\in\mathbb{R}$,
\begin{align*}
|\psi(t_1,x_1,v_1)-\psi(t_2,x_2,v_2)|\leq C(1+|x_1|+|x_2|)e^{C(|x_1|+|x_2|)}(|t_1-t_2|+|x_1-x_2|+|v_1-v_2|),
\end{align*}
where $\psi=\partial_t\varphi,\partial_x\varphi, \partial_{xx}^2\varphi$ and $\partial_v\varphi$.
Let $f$, $h$ and $g$ be bounded processes in $M_{G}^2(0,T)$. If for each $p\geq 1$, the continuous finite variation process $V_t\in{L}_b^p(\Omega_t)$
and there exists some constant $C_p$ such that for each $s\leq t\in[0,T]$:
\[
\hat{\mathbb{E}}[|T^{V}_t-T^V_s|^p]\leq C_p|t-s|^p,
\]
where $T^V$ is the total variation process of $V$.
Then  in ${L}_b^2(\Omega_t)$,
\begin{align*}
\varphi(t,X_t,V_t)=&\varphi(0,X_0,V_0)+\int^t_0\partial_u \varphi(u,X_{u},V_u) du+\int^t_0\partial_x \varphi(u,X_{u},V_u)f_u du\\
&+\int^t_0\partial_x \varphi(u,X_{u},V_u)h_u d\langle B\rangle_u+\int^t_0\partial_x \varphi(u,X_{u},V_u)g_u dB_u\\
&+\int^t_0\partial_v\varphi(u,X_{u},V_u) dV_u+\frac{1}{2}\int^t_0\partial^2_{xx}\varphi(u,X_{u},V_u)g_{u}^2 d\langle B\rangle_u.
\end{align*}
\end{theorem}
\begin{proof}
The proof is immediate in light of Lemma \ref{lem5}, Theorem 5.4 of Li and Peng \cite{LP} and Theorem 2.33 of Lin \cite{LY}.
\end{proof}

\textbf{Acknowledgement}: The authors would like to thank Prof. Peng, S. for   his helpful discussions and
suggestions. The authors also thank the editor and two anonymous referees for their careful reading,  helpful
suggestions.


\end{document}